\documentclass[10pt]{article}%
\usepackage{amsmath}
\usepackage{amsfonts}
\usepackage{mathrsfs}
\usepackage{amssymb}
\usepackage[linkcolor=black,anchorcolor=black,citecolor=black]{hyperref}
\usepackage{graphicx}
\usepackage[body={15.5cm,21cm}, top=3cm]{geometry}%
\providecommand{\U}[1]{\protect\rule{.1in}{.1in}}
\providecommand{\U}[1]{\protect \rule{.1in}{.1in}}
\newtheorem{theorem}{Theorem}[section]

\newtheorem{corollary}[theorem]{Corollary}

\newtheorem{definition}[theorem]{Definition}
\newtheorem{example}[theorem]{Example}

\newtheorem{lemma}[theorem]{Lemma}

\newtheorem{proposition}[theorem]{Proposition}
\newtheorem{remark}[theorem]{Remark}

\newenvironment{proof}[1][Proof]{\noindent \textbf{#1.} }{\  \rule{0.5em}{0.5em}}
\begin{document}

\title{Invariant and ergodic measures for $G$-diffusion processes}
\author{Mingshang Hu \thanks{Qilu Institute of Finance, Shandong University,
humingshang@sdu.edu.cn. Research supported by NSF (No. 11201262, 11101242 and 11301068) and Shandong Province (No. BS2013SF020)}
\and Hanwu Li\thanks{School of Mathematics, Shandong University,
lihanwu11@163.com.}
\and Falei Wang\thanks{School of Mathematics, Shandong University,
flwang2011@gmail.com.}\and Guoqiang Zheng\thanks{School of Mathematics, Shandong University,  zhengguoqiang.ori@gmail.com. Hu, Li, Wang and Zheng's research was
partially supported by NSF (No. 10921101) and by the 111
Project (No. B12023)}}
\date{}

\maketitle
\begin{abstract}
In this paper we study the problems of invariant and ergodic measures under $G$-expectation framework.
 In particular,    the stochastic differential
equations driven by $G$-Brownian motion ($G$-SDEs) have the   unique invariant and ergodic measures.
Moreover, the invariant and ergodic measures of $G$-SDEs are also sublinear expectations.
However, the invariant measures may not coincide with ergodic measures, which is different from the classical case.
\end{abstract}


\textbf{MSC-classification}: 60H10, 60H30

\section{Introduction}
Recently, Peng systemically established a time-consistent fully nonlinear
expectation theory (see \cite {Peng2005, P10, PengICM2010} and the references therein),
which is an effective  tool to study the problems of
model uncertainty, nonlinear stochastic dynamical systems and fully nonlinear partial differential equations (PDEs).
As a typical and important case, Peng  introduced the
$G$-expectation theory. In the  $G$-expectation framework, the
notion of  $G$-Brownian motion and the corresponding stochastic
calculus of It\^{o}'s type were also established. Moreover, Peng \cite {P10} and Gao \cite{G} obtained the existence and uniqueness theorem  of $G$-SDEs.

It is well known that invariant measure plays an important role in the theory of stochastic dynamical systems and ergodic theory.
In particular, the invariant measure can be thought of as describing the long-term
behaviour of a dynamical system, which  has many important applications in,
for example,  PDEs and  financial mathematics.
By far,
there are many papers in the literature which were devoted to study the invariant measures of Markov processes, both in finite and infinite
dimension spaces (see \cite {DZ} and the references therein).

The aim of this paper is to  study the asymptotic property of $G$-SDEs.
 First, we  obtain the existence and uniqueness theorem of invariant
measures for $G$-SDEs. The proof of the existence theorem is based on Daniell-Stone Theorem.
It is important to point out that the standard
techniques and results on invariant measures for Markov processes cannot be
applied to deal with this problem  because $G$-expectation is not
a linear expectation.
Under $G$-expectation framework, the invariant measure of $G$-SDE is a family of probability measures.   In particular, if the initial condition has the distribution equal to an
invariant measure, then the distribution of the solution to $G$-SDE is invariant in time as the classical case.
Next, we study the ergodicity of  $G$-SDEs.
Under nonlinear case,  the ergodic measure of $G$-SDE may not be
 the corresponding invariant measure.
 The proof of the existence theorem of  ergodic measure  is based on the theory of ergodic backward differential equations driven by $G$-Brownian motion, which is obtained in \cite{HW} (see also
 \cite{DH,H,R,RM}).

The paper is organized as follows. In  section 2, we present some notations and results which will be  used in this paper.
The  existence and uniqueness theorem of invariant measures of $G$-diffusion processes is established in section 3.
In section 4, we shall study the relationships between invariant measures and  ergodic measures under the $G$-expectation framework.

\section{Preliminaries}

The main purpose of this section is to recall some basic notions and results of $G$-expectation, which are needed in the sequel. The readers may refer to \cite{HJPS}, \cite{HJPS1},
\cite{P07a}, \cite{P08a}, \cite{P10} for more details.

\begin{definition}
\label{def2.1} Let $\Omega$ be a given set and let $\mathcal{H}$ be a vector
lattice of real valued functions defined on $\Omega$, namely $c\in \mathcal{H%
}$ for each constant $c$ and $|X|\in \mathcal{H}$ if $X\in \mathcal{H}$. $%
\mathcal{H}$ is considered as the space of random variables. A sublinear
expectation $\mathbb{\hat{E}}$ on $\mathcal{H}$ is a functional $\mathbb{%
\hat {E}}:\mathcal{H}\rightarrow \mathbb{R}$ satisfying the following
properties: for all $X,Y\in \mathcal{H}$, we have

\begin{description}
\item[(a)] Monotonicity: If $X\geq Y$ then $\mathbb{\hat{E}}[X]\geq \mathbb{%
\hat{E}}[Y]$;

\item[(b)] Constant preservation: $\mathbb{\hat{E}}[c]=c$;

\item[(c)] Sub-additivity: $\mathbb{\hat{E}}[X+Y]\leq \mathbb{\hat{E}}[X]+%
\mathbb{\hat{E}}[Y]$;

\item[(d)] Positive homogeneity: $\mathbb{\hat{E}}[\lambda X]=\lambda
\mathbb{\hat{E}}[X]$ for each $\lambda \geq0$.
\end{description}

The triple $(\Omega,\mathcal{H},\hat{\mathbb{E}})$ is called a
sublinear expectation space.  $X \in\mathcal{ H}$ is called a random
variable in $(\Omega,\mathcal{H},\hat{\mathbb{E}})$. We often call
$Y = (Y_1, \ldots, Y_d), Y_i \in\mathcal{ H}$ a $d$-dimensional
random vector in $(\Omega,\mathcal{H},\hat{\mathbb{E}})$.
\end{definition}

\begin{definition}
\label{def2.2} Let $X_{1}$ and $X_{2}$ be two $n$-dimensional random vectors
defined respectively in sublinear expectation spaces $(\Omega_{1},\mathcal{H}%
_{1},\mathbb{\hat{E}}_{1})$ and $(\Omega_{2},\mathcal{H}_{2},\mathbb{\hat{E}}%
_{2})$. They are called identically distributed, denoted by $X_{1}\overset{d}%
{=}X_{2}$, if $\mathbb{\hat{E}}_{1}[\varphi(X_{1})]=\mathbb{\hat{E}}%
_{2}[\varphi(X_{2})]$, for all$\ \varphi \in C_{Lip}(\mathbb{R}^{n})$,
where $C_{Lip}(\mathbb{R}^{n})$ is the space of real $\mathbb{R}$-valued Lipschitz continuous functions
defined on $\mathbb{R}^{n}$.
\end{definition}

\begin{definition}
\label{def2.3} In a sublinear expectation space $(\Omega,\mathcal{H},\mathbb{%
\hat{E}})$, a random vector $Y=(Y_{1},\cdot \cdot \cdot,Y_{n})$, $Y_{i}\in
\mathcal{H}$, is said to be independent of another random vector $%
X=(X_{1},\cdot \cdot \cdot,X_{m})$, $X_{i}\in \mathcal{H}$ under $\mathbb{%
\hat {E}}[\cdot]$, denoted by $Y\bot X$, if for every test function $\varphi
\in C_{Lip}(\mathbb{R}^{m}\times \mathbb{R}^{n})$ we have $\mathbb{\hat{E}}%
[\varphi(X,Y)]=\mathbb{\hat{E}}[\mathbb{\hat{E}}[\varphi(x,Y)]_{x=X}]$.
\end{definition}

\begin{definition}
\label{def2.4} ($G$-normal distribution) A $d$-dimensional random vector $%
X=(X_{1},\cdot \cdot \cdot,X_{d})$ in a sublinear expectation space $(\Omega,%
\mathcal{H},\mathbb{\hat{E}})$ is called $G$-normally distributed if for
each $a,b\geq0$ we have
\begin{equation*}
aX+b\bar{X}\overset{d}{=}\sqrt{a^{2}+b^{2}}X,
\end{equation*}
where $\bar{X}$ is an independent copy of $X$, i.e., $\bar{X}\overset{d}{=}X$
and $\bar{X}\bot X$. Here the letter $G$ denotes the function
\begin{equation*}
G(A):=\frac{1}{2}\mathbb{\hat{E}}[\langle AX,X\rangle]:\mathbb{S}%
_{d}\rightarrow \mathbb{R},
\end{equation*}
where $\mathbb{S}_{d}$ denotes the collection of $d\times d$ symmetric
matrices.
\end{definition}

Let $\Omega=C_{0}([0,\infty);\mathbb{R}^{d})$, the space of
$\mathbb{R}^{d}$-valued continuous functions on $[0,\infty)$ with $\omega_{0}=0$, be endowed
with the distance
$$
\rho(\omega^1, \omega^2):=\sum^\infty_{N=1} 2^{-N} [(\max_
{t\in[0,N]} | \omega^1_t-\omega^2_t|)  \wedge 1],
$$
and $B=(B^i)_{i=1}^d$ be the canonical
process. For each $T>0$, denote
\[
L_{ip} (\Omega_T):=\{ \varphi(B_{t_{1}},...,B_{t_{n}}):n\geq1,t_{1}%
,...,t_{n}\in\lbrack0,T],\varphi\in C_{Lip}(\mathbb{R}^{d\times n})\}, \ L_{ip} (\Omega):=\underset{T}{\cup}L_{ip} (\Omega_T).
\]
For any given monotonic and sublinear function
$G:\mathbb{S}_{d}\rightarrow\mathbb{R}$, let  $(\Omega,  L_{ip} (\Omega),\mathbb{\hat{E}},\mathbb{\hat{E}}_t)$ be the $G$-expectation space,
where
$G(A)=\frac{1}{2}\mathbb{\hat{E}}[\langle AB_1,B_1\rangle]\leq \frac{1}{2}\bar{\sigma}^2|A|$.

 Denote by $L_{G}^{p}(\Omega)$   the completion of
$L_{ip} (\Omega)$ under the norm $\Vert\xi\Vert_{L_{G}^{p}}:=(\mathbb{\hat{E}}[|\xi|^{p}])^{1/p}$ for $p\geq1$.
Denis et al. \cite{DHP11}
proved that the completions of $C_{b}(\Omega)$ (the set of bounded
continuous function on $\Omega$) and $L_{ip} (\Omega)$ under $\Vert\cdot\Vert_{L_{G}^{p}}$ are the same. Similarly, we can define $L_{G}^{p}(\Omega_T)$ for each $T>0$.
\begin{theorem}[\cite{DHP11,HP09}]
\label{the2.7}  There exists a weakly compact set
$\mathcal{P}\subset\mathcal{M}_{1}(\Omega)$, the set of all probability
measures on $(\Omega,\mathcal{B}(\Omega))$, such that
\[
\mathbb{\hat{E}}[\xi]=\sup_{P\in\mathcal{P}}E_{P}[\xi]\ \ \text{for
\ all}\ \xi\in  {L}_{G}^{1}{(\Omega)}.
\]
$\mathcal{P}$ is called a set that represents $\mathbb{\hat{E}}$.
\end{theorem}

Let $\mathcal{P}$ be a weakly compact set that represents $\mathbb{\hat{E}}$.
For this $\mathcal{P}$, we define capacity%
\[
c(A):=\sup_{P\in\mathcal{P}}P(A),\ A\in\mathcal{B}(\Omega).
\]
A set $A\subset\mathcal{B}(\Omega)$ is polar if $c(A)=0$.  A
property holds $``quasi$-$surely''$ (q.s.) if it holds outside a
polar set. In the following, we do not distinguish two random
variables $X$ and $Y$ if $X=Y$ q.s..

\begin{definition}
\label{def2.6} Let $M_{G}^{0}(0,T)$ be the collection of processes in the
following form: for a given partition $\{t_{0},\cdot\cdot\cdot,t_{N}\}=\pi
_{T}$ of $[0,T]$,
\[
\eta_{t}(\omega)=\sum_{j=0}^{N-1}\xi_{j}(\omega)\mathbf{1}_{[t_{j},t_{j+1})}(t),
\]
where $\xi_{i}\in L_{ip}(\Omega_{t_{i}})$, $i=0,1,2,\cdot\cdot\cdot,N-1$. For each
$p\geq1$,  denote by  $M_{G}^{p}(0,T)$ the completion
of $M_{G}^{0}(0,T)$ under the norm $\Vert\eta\Vert_{M_{G}^{p}}:=(\mathbb{\hat{E}}[\int_{0}^{T}|\eta_{s}|^{p}ds])^{1/p}$.
\end{definition}

For two processes $ \eta\in M_{G}^{2}(0,T)$ and $ \xi\in M_{G}^{1}(0,T)$,
the $G$-It\^{o} integrals  $(\int^{t}_0\eta_sdB^i_s)_{0\leq t\leq T}$ and $(\int^{t}_0\xi_sd\langle
B^i,B^j\rangle_s)_{0\leq t\leq T}$  are well defined, see  Li-Peng \cite{L-P} and Peng \cite{P10}.

\section{Invariant measures}
In this section, we shall study the invariant measures of $G$-diffusion processes. Let $G:\mathbb{S}%
_{d}\rightarrow\mathbb{R}$ be a given monotonic and sublinear function
and $B_{t}=(B_{t}^{i})_{i=1}^{d}$ be the corresponding $d$-dimensional $G$-Brownian motion.
For a given integer $p\geq 1$, a real-valued function $f$ defined on $\mathbb{R}^n$ is said to be in $C_{p,Lip}(\mathbb{R}^n)$ if there exists a constant $K_f$ depending on $f$ such that
$|f(x)-f(x^{\prime})|\leq K_f(1+|x|^{p-1}+|x^{\prime}|^{p-1})|x-x^{\prime}|.$
Consider the following type of $G$-SDEs (in this paper we always use Einstein convention): for each $t\geq 0$ and $\xi\in L_{G}^{m}(\Omega_{t})$ with $m\geq2$,
\begin{align} \label{App1}
X_{s}^{t,\xi}=\xi+\int^s_tb(X_{r}^{t,\xi})dr+\int^s_th_{ij}(X_{r}^{t,\xi})d\langle
B^i,B^j\rangle_{r}+\int^s_t\sigma(X_{r}^{t,\xi})dB_{r},
\end{align}
where $b$, $h_{ij}:\mathbb{R}^{n}\rightarrow \mathbb{R}^{n}$, $\sigma
:\mathbb{R}^{n}\rightarrow \mathbb{R}^{n\times d}$ are deterministic continuous
functions. In particular, denote $X^x=X^{0,x}$.
Consider also the following assumptions:
\begin{description}
\item[(H1)] There exists a constant $L>0$ such that%
\begin{align*}
&|b(x)-b(x^{\prime})|+\sum\limits_{i,j}|h_{ij}(x)-h_{ij}(x^{\prime
})|+ |\sigma(x)-\sigma(x^{\prime})|\leq
L|x-x^{\prime}|.
\end{align*}
\item[(H2)] $G((2p-1)\sum_{i=1}^{n}(\sigma_{i}(x)-\sigma_{i}(x^{\prime}%
))^{T}(\sigma_{i}(x)-\sigma_{i}(x^{\prime}))+2(\langle x-x^{\prime}%
,h_{ij}(x)-h_{ij}(x^{\prime})\rangle)_{i,j=1}^{d})+\langle x-x^{\prime
},b(x)-b(x^{\prime})\rangle \leq-\eta|x-x^{\prime}|^{2}$ for some constants
$\eta>0$, where $\sigma_{i}$ is the $i$-th row of $\sigma$.
\end{description}

We have the following estimates of $G$-SDEs which can be found in Chapter V in
Peng \cite{P10}.

\begin{lemma}
\label{proA.1}
Under assumption \emph{(H1)},
the $G$-SDE \eqref{App1} has a unique solution $X^{t,\xi}\in M^2_G(t,T)$ for each $T>t$.
Moreover, if $\xi$, $\xi^{\prime}\in L_{G}^{m}(\Omega_{t})$ with $m\geq2$, then we have, for each $\delta\in\lbrack0,T-t]$,%
\begin{description}
\item[(i)] $\mathbb{\hat{E}}_{t}[\sup\limits_{s\in\lbrack t,T]}|X_{s}^{t,\xi}-X_{s}^{t,\xi^{\prime}}%
|^{m}]\leq C^{\prime}|\xi-\xi^{\prime}|^{m}$;
\item[(ii)] $\mathbb{\hat{E}}_{t}[\sup\limits_{s\in\lbrack t,T]}|X_{s}^{t,\xi}|^{m}]\leq C^{\prime}(1+|\xi|^{m})$;
\item[(iii)] $\mathbb{\hat{E}}_{t}[\sup\limits_{s\in\lbrack t,t+\delta]}|X_{s}^{t,\xi}-\xi
|^{m}]\leq C^{\prime}(1+|\xi|^{m})\delta^{m/2}$,
\end{description}
where the constant $C^{\prime}$ depends on $L$, $G$, $m$, $n$ and $T$.
\end{lemma}

The following result is important in our future discussion (see also \cite{HW}). Specially, the constant $C$ is independent of $T$.
\begin{lemma}\label{HW2}
 Under assumptions \emph{(H1)} and \emph{(H2)}, if $\xi$, $\xi^{\prime}\in L_{G}^{2p}(\Omega_{t})$,
 then there exists a constant ${C}$ depending on $G, L, p, n$ and $\eta$, such that:
\begin{description}
\item[(i)] $\hat{\mathbb{E}}_t[|X_{s}^{t,\xi}-X_{s}^{t,\xi^{\prime}}|^{2p}]\leq\exp(-2\eta p(s-t))|\xi-\xi^{\prime}|^{2p}$;
\item[(ii)] $\hat{\mathbb{E}}_t[|X_{s}^{t,\xi}|^{2p}]\leq {C}(1+|\xi|^{2p}), \ \forall t>0$.
\end{description}
\end{lemma}
\begin{proof}
To simplify presentation, we shall prove only the case when $n=d=1$, as the
higher dimensional case can be treated in the same way without
difficulty. Set
$C_{s}:=\exp(2p\eta(s-t))$. Applying the $G$-It\^{o} formula yields that
\begin{align*}
&  C_{s}(X_{s}^{t,\xi}-X_{s}^{t,{\xi}^{\prime}})^{2p}-|\xi-{\xi}^{\prime
}|^{2p}\\
&  =2p\eta \int_{t}^{s}C_{r}(X_{r}^{t,\xi}-X_{r}^{t,{\xi}^{\prime}}%
)^{2p}dr+2p\int_{t}^{s}C_{r}(X_{r}^{t,\xi}-X_{r}^{t,{\xi}^{\prime}}%
)^{2p-1}(b(X_{r}^{t,\xi})-b(X_{r}^{t,{\xi}^{\prime}}))dr\\
&  \  \ +p\int_{t}^{s}\xi_{r}d\langle B\rangle_{r}+2p\int_{t}^{s}C_{r}%
(X_{r}^{t,\xi}-X_{r}^{t,{\xi}^{\prime}})^{2p-1}(\sigma(X_{r}^{t,\xi}%
)-\sigma(X_{r}^{t,{\xi}^{\prime}}))dB_{r}\\
&  =2p\eta \int_{t}^{s}C_{r}(X_{r}^{t,\xi}-X_{r}^{t,{\xi}^{\prime}}%
)^{2p}dr+2p\int_{t}^{s}C_{r}(X_{r}^{t,\xi}-X_{r}^{t,{\xi}^{\prime}}%
)^{2p-1}(b(X_{r}^{t,\xi})-b(X_{r}^{t,{\xi}^{\prime}}))dr\\
&  \  \ +2p\int_{t}^{s}G(\xi_{r})dr+2p\int_{t}^{s}C_{r}(X_{r}^{t,\xi}%
-X_{r}^{t,{\xi}^{\prime}})^{2p-1}(\sigma(X_{r}^{t,\xi})-\sigma(X_{r}^{t,{\xi
}^{\prime}}))dB_{r}\\
&  \  \ +p\int_{t}^{s}\xi_{r}d\langle B\rangle_{r}-2p\int_{t}^{s}G(\xi_{r})dr,
\end{align*}
where
\[
\xi_{r}=C_{r}(X_{r}^{t,\xi}-X_{r}^{t,{\xi}^{\prime}})^{2p-2}((2p-1)|\sigma
(X_{r}^{t,\xi})-\sigma(X_{r}^{t,{\xi}^{\prime}})|^{2}+2(X_{r}^{t,\xi}%
-X_{r}^{t,{\xi}^{\prime}})(h(X_{r}^{t,\xi})-h(X_{r}^{t,{\xi}^{\prime}}))).
\]
Note that $\int_{t}^{s}\xi_{r}d\langle B\rangle_{r}-2\int_{t}^{s}G(\xi
_{r})dr\leq0$ and {(H2)}, then we obtain%
\begin{align}\label{myq1}
C_{s}(X_{s}^{t,\xi}-X_{s}^{t,{\xi}^{\prime}})^{2p}-|\xi-{\xi}^{\prime}%
|^{2p}\leq2p\int_{t}^{s}C_{r}(X_{r}^{t,\xi}-X_{r}^{t,{\xi}^{\prime}}%
)^{2p-1}(\sigma(X_{r}^{t,\xi})-\sigma(X_{r}^{t,{\xi}^{\prime}}))dB_{r}.
\end{align}
On the other hand, by Lemma \ref{proA.1},%
\[
\hat{\mathbb{E}}[(\int_{t}^{T}|X_{r}^{t,\xi}-X_{r}^{t,{\xi}^{\prime}}%
|^{4p}dr)^{1/2}]\leq \sqrt{T}\hat{\mathbb{E}}[\sup_{r\in \lbrack t,T]}%
|X_{r}^{t,\xi}-X_{r}^{t,{\xi}^{\prime}}|^{2p}]\leq \sqrt{T}C^{\prime}%
\hat{\mathbb{E}}[|\xi-\xi^{\prime}|^{2p}].
\]
Then the right sides of inequality \eqref{myq1} is a $G$-martingale.
Thus we conclude that
\[
\mathbb{\hat{E}}_{t}[C_{s}|X_{s}^{t,\xi}-X_{s}^{t,{\xi}^{\prime}}|^{2p}%
]\leq|\xi-{\xi}^{\prime}|^{2p}.
\]
Consequently,
\[
\mathbb{\hat{E}}_{t}[|X_{s}^{t,\xi}-X_{s}^{t,{\xi}^{\prime}}|^{2p}]\leq
\exp(-2p\eta(s-t))|\xi-{\xi}^{\prime}|^{2p}.
\]
By a similar analysis as in of Lemma 4.1 of \cite{HW}, we can also obtain the
second inequality holds, which completes the proof.
\end{proof}

\begin{theorem}\label{HW3}
Assume \emph{(H1)} and \emph{(H2)} hold. Then for each  $f\in C_{2p,Lip}(\mathbb{R}^n)$, there exists a constant $\bar{\lambda}^f$ such that
\[
\lim\limits_{t\rightarrow\infty}\hat{\mathbb{E}}[f(X^{x}_{t})]=\bar{\lambda}^f, \ \ \forall x\in\mathbb{R}^n.
\]
In particular,  for each $t$, there exists a constant $C_1$ depending on $G$, $\eta, L, K_f, n$ and $p$ such that
\[
|\bar{\lambda}^f-\hat{\mathbb{E}}[f(X^{x}_{t})]|\leq {C_1}(1+|x|^{2p})\exp(-\eta t).
\]
\end{theorem}
\begin{proof}
For a fixed $x$ and each $f\in C_{2p,Lip}(\mathbb{R}^n)$,
 from Lemma \ref{HW2}, we can find some constant $\bar{C}$ depending on $C$ and $K_f$ such that
\[\hat{\mathbb{E}}[|f(X^x_t)|]\leq |f(0)|+\bar{C}\hat{\mathbb{E}}[|X^x_t|^{2p}]\leq \bar{C}(1+|x|^{2p}).
\]
Then there exists a sequence $T_n\rightarrow\infty$  such that
$\hat{\mathbb{E}}[f(X^x_{T_n})]\rightarrow \bar{\lambda}^f$  for some constant $\bar{\lambda}^f$.
From the uniqueness of solutions to $G$-SDEs, we obtain $X^x_s=X^{t,X^x_t}_s$ with $s\geq t$.
Note that $\hat{\mathbb{E}}[f(X^{x}_{t^{\prime}})]=\hat{\mathbb{E}}[f(X^{t-t^{\prime},x}_t)]$ for each $t$ and $t^{\prime}$ with $t^{\prime}\leq t$, then we have
\begin{align*}
|\hat{\mathbb{E}}[f(X^x_t)]-\hat{\mathbb{E}}[f(X^{x}_{t^{\prime}})]|=&|\hat{\mathbb{E}}[f(X^{t-t^{\prime},X^x_{t-t^{\prime}}}_t)]-\hat{\mathbb{E}}[f(X^{t-t^{\prime},x}_t)]|
\\ \leq & K_f \hat{\mathbb{E}}[(1+|X^{t-t^{\prime},X^x_{t-t^{\prime}}}_t|^{2p-1}+|X^{t-t^{\prime},x}_t|^{2p-1})|X^{t-t^{\prime},X^x_{t-t^{\prime}}}_t-X^{t-t^{\prime},x}_t|].
\end{align*}
Applying H\"{o}lder's inequality and Lemma \ref{HW2}, we obtain that
 \begin{align*}
|\hat{\mathbb{E}}[f(X^x_t)]-\hat{\mathbb{E}}[f(X^{x}_{t^{\prime}})]|&\leq K_f \hat{\mathbb{E}}[(1+|X^{t-t^{\prime},X^x_{t-t^{\prime}}}_t|^{2p-1}+|X^{t-t^{\prime},x}_t|^{2p-1})^{\frac{2p}{2p-1}}]^{\frac{2p-1}{2p}}\hat{\mathbb{E}}[|X^{t-t^{\prime},X^x_{t-t^{\prime}}}_t-X^{t-t^{\prime},x}_t|^{2p}]^{\frac{1}{2p}}
\\ &\leq C_1 \hat{\mathbb{E}}[1+|X^{t-t^{\prime},X^x_{t-t^{\prime}}}_t|^{2p}+|X^{t-t^{\prime},x}_t|^{2p}]^{\frac{2p-1}{2p}}\hat{\mathbb{E}}[|X^x_{t-t^{\prime}}|^{2p}+|x|^{2p}]^{\frac{1}{2p}}\exp(-\eta t^{\prime})\\
& \leq C_1 (1+|x|^{2p})^{\frac{1}{2p}}\hat{\mathbb{E}}[\hat{\mathbb{E}}_{t-t^{\prime}}[1+|X^{t-t^{\prime},X^x_{t-t^{\prime}}}_t|^{2p}+|X^{t-t^{\prime},x}_t|^{2p}]]^{\frac{2p-1}{2p}}\exp(-\eta t^{\prime})
\\ &\leq C_1 (1+|x|^{2p})^{\frac{1}{2p}}\hat{\mathbb{E}}[1+|X^x_{t-t^{\prime}}|^{2p}+|x|^{2p}]^{\frac{2p-1}{2p}}\exp(-\eta t^{\prime})
\\& \leq C_1 (1+|x|^{2p})\exp(-\eta t^{\prime}),
\end{align*}
where the constant $C_1$  depending on $p$ and $G,\eta, n, L, K_f$ is  vary from line to line.

Consequently, for each $t$, we get
\[
|\bar{\lambda}^f-\hat{\mathbb{E}}[f(X^{x}_{t})]|=\lim\limits_{n\rightarrow\infty}|\hat{\mathbb{E}}[f(X^x_{T_n})]-\hat{\mathbb{E}}[f(X^{x}_{t})]|\leq {C_1}(1+|x|^{2p})\exp(-\eta t),
\]
which derives that
\[
\bar{\lambda}^f=\lim\limits_{t\rightarrow\infty}\hat{\mathbb{E}}[f(X^{x}_{t})].
\]
For each $x,x^{\prime}\in\mathbb{R}^n$, applying Lemma  \ref{HW2} (i) yields that
\begin{align*}
\lim\limits_{t\rightarrow\infty}|\hat{\mathbb{E}}[f(X^{x}_{t})]-\hat{\mathbb{E}}[f(X^{x^{\prime}}_{t})]|\leq &
\lim\limits_{t\rightarrow\infty}\hat{\mathbb{E}}[|f(X^{x}_{t})-f(X^{x^{\prime}}_{t})|]\\
\leq & K_f
\lim\limits_{t\rightarrow\infty}\hat{\mathbb{E}}[(1+|X^x_t|^{2p-1}+|X^{x^{\prime}}_t|^{2p-1})|X^{x}_{t}-X^{x^{\prime}}_{t}|]\\
\leq &
K_f \lim\limits_{t\rightarrow\infty}\hat{\mathbb{E}}[(1+|X^{x}_t|^{2p-1}+|X^{x^{\prime}}_t|^{2p-1})^{\frac{2p}{2p-1}}]^{\frac{2p-1}{2p}}\hat{\mathbb{E}}[|X^{x}_t-X^{x^{\prime}}_t|^{2p}]^{\frac{1}{2p}}\\
\leq & C_1 \lim\limits_{t\rightarrow\infty}(1+|x|^{2p}+|x^{\prime}|^{2p})\exp(-\eta t)=0,
\end{align*}
which completes the proof.
\end{proof}

The following result is a direct consequence of Theorem \ref{HW3}.
\begin{corollary}\label{HW8}
For each $f\in  C_{2p,Lip}(\mathbb{R}^n)$, we get
\[
\lim\limits_{T\rightarrow\infty}\frac{1}{T}\int^T_0\mathbb{\hat{E}}[f(X^x_t)]dt=\bar{\lambda}^f, \ \ \forall x\in\mathbb{R}^n.
\]
\end{corollary}

From the nonlinear Feynman-Kac formula in \cite{P10},  we obtain $u^f(t,x)=\hat{\mathbb{E}}[f(X^x_t)]$ is the unique viscosity solution to the following fully nonlinear PDE.
 \begin{align}\label{W1}
 \begin{cases}
&\partial_tu^f-G(H(D_{x}^{2}u^f,D_{x}u^f,x))-\langle
b(x),D_{x}u^f\rangle=0,\ (t,x)\in(0,\infty)\times\mathbb{R}^n,\\
& u^f(0,x)=f(x).
\end{cases}
\end{align}
where
\begin{align*}
H_{ij}(D_{x}^{2}u^f,D_{x}u^f,x)=   \langle D_{x}^{2}u^f\sigma_{i}%
(x),\sigma_{j}(x)\rangle+2\langle D_{x}u^f,h_{ij}(x)\rangle.
\end{align*}
Then by Lemma \ref{HW3}, we  get the following large time behaviour of  solution to fully nonlinear parabolic  PDE \eqref{W1}.
\begin{corollary}
For each $f\in  C_{2p,Lip}(\mathbb{R}^n)$, we have for any $x\in\mathbb{R}^n$,
\[
\lim\limits_{T\rightarrow\infty}u^f(T,x)=\bar{\lambda}^f \ \text {and} \ |u^f(T,x)-\bar{\lambda}^f|\leq C_1 (1+|x|^{2p})\exp(-\eta T).
\]
\end{corollary}

We define the function $\bar{\Lambda}: C_{2p,Lip}(\mathbb{R}^n)\mapsto\mathbb{R}$ by \[\bar{\Lambda}[f]=\bar{\lambda}^f.\]
\begin{lemma}
Assume \emph{(H1)} and \emph{(H2)} hold. Then  $\bar{\Lambda}$ is a sublinear expectation on $(\mathbb{R}^n, C_{2p,Lip}(\mathbb{R}^n))$, i.e.,
\begin{description}
\item[(a)] If $f_1\geq f_2$, then $\bar{\Lambda}[f_1]\geq
\bar{\Lambda}[f_2]$;
\item[(b)]  $\bar{\Lambda}[c]=c$ for any constant $c$;
\item[(c)]  $\bar{\Lambda}[f_1+f_2]\leq\bar{\Lambda}%
[f_1]+\bar{\Lambda}[f_2]$;
\item[(d)] $\bar{\Lambda}[\lambda f]=\lambda
\bar{\Lambda}[f]$ for each $\lambda\geq0$.
\end{description}
\end{lemma}
\begin{proof}
The proof is immediate from Theorem \ref{HW3} and the definition of  $G$-expectation.
\end{proof}
\begin{lemma}\label{newds}
For each sequence
$\{f _{i}\}_{i=1}^{\infty }\subset C_{2p-1,Lip}(\mathbb{R}^{n})$
satisfying $f_{i}\downarrow 0$, we have
$\bar{\Lambda}[f_{i}]\downarrow 0$.
\end{lemma}
\begin{proof}
For each fixed $N>0$,
\begin{equation*}
f_{i}(x)\leq k_{i}^{N}+f_{1}(x)\mathbf{1}_{[|x|>N]}\leq k_{i}^{N}+\frac{%
f_{1}(x)|x|}{N}\text{\ for every }x\in \mathbb{R}^{n},
\end{equation*}%
where $k_{i}^{N}=\max_{|x|\leq N}f_{i}(x)$.
Then we   have,
\begin{equation*}
\hat{\mathbb{E}}[f_{i}(X^x_t)]\leq k_{i}^{N}+\frac{1}{N}\hat{\mathbb{E}}[f_{1}(X^x_t)|X^x_t|].
\end{equation*}
Applying Lemma \ref{HW2}, there exits a constant $C_1$ depending on $G$, $f_1$, $p, n$ and $\eta$ such that,
$$\hat{\mathbb{E}}[f_{1}(X^x_t)|X^x_t|]\leq \bar{C}\hat{\mathbb{E}}[|f_1(0)X^x_t|+|X^x_t|^{2p}]\leq C_1(1+|x|^{2p}).$$
Consequently,\[
\bar{\Lambda}[f_i]=\lim\limits_{t\rightarrow\infty}\hat{\mathbb{E}}[f_{i}(X^x_t)]\leq k_{i}^{N}+\frac{C_1(1+|x|^{2p})}{N}.
\]
It follows from $f_{i}\downarrow 0$ and Dini's theorem that $k_{i}^{N}\downarrow
0$. Thus we have
$\lim_{i\rightarrow \infty
}\bar{\Lambda}[f_i]\leq
\frac{C_1(1+|x|^{2p})}{N}$. Since $N$ can
be arbitrarily large, we get $\bar{\Lambda}[f_{i}]\downarrow 0$.
\end{proof}
\begin{remark}{\upshape
From the above proof, in general we cannot get this result  for $\{f_i\}_{i=1}^{\infty}\subset C_{2p,Lip}(\mathbb{R}^{n})$.
}
\end{remark}

\begin{theorem}\label{HW4}
Suppose assumptions \emph{(H1)} and \emph{(H2)} hold.
Then there exists  a family of weakly compact  probability measures $\{m_{\theta}\}_{\theta\in \bar{\Theta}}$ defined on $(\mathbb{R}^n, \mathcal{B}(\mathbb{R}^n))$ such that
\[
\bar{\lambda}^f=\sup\limits_{\theta\in \bar{\Theta}}\int_{\mathbb{R}^n}f(x)m_{\theta}(dx), \ \ \forall f\in C_{2p-1,Lip}(\mathbb{R}^n).
\]
\end{theorem}
\begin{proof}
By the representation theorem (Theorem 2.1 of  Chapter 1 in \cite{P10}),
for the sublinear expectation $\bar{\Lambda}[f]$ defined on
$(\mathbb{R}^{n},C_{2p-1,Lip}(\mathbb{R}^{n}))$, there exists a family
of linear expectations $\{M_{\theta }\}_{\theta \in \hat{\Theta} }$ on
$(\mathbb{R}^{n},C_{2p-1,Lip}(\mathbb{R}^{n}))$ such that
\begin{equation*}
\bar{\Lambda}[f ]=\sup_{\theta \in \hat{\Theta} }M_{\theta
}[f ],\ \ \forall f \in C_{2p-1,Lip}(\mathbb{R}^{n}).
\end{equation*}%
By Lemma \ref{newds}, for each sequence $\{f _{i}\}_{i=1}^{\infty }$ in $%
C_{2p-1,Lip}(\mathbb{R}^{n})$ such that $f_{i}\downarrow 0$ on $\mathbb{R%
}^{n}$, we have $\bar{\Lambda}[f_{i}]\downarrow 0$. Thus $M_{\theta
}[f_{i}]\downarrow 0$ for each $\theta \in \hat{\Theta}$. It
follows from the Daniell-Stone Theorem that, for
each $\theta \in \hat{\Theta} $, there exists a unique probability measure $%
m_{\theta }(\cdot )$ on $(\mathbb{R}^{n},\sigma (C_{2p-1,Lip}(\mathbb{R}^{n}))=(%
\mathbb{R}^{n},\mathcal{B}(\mathbb{R}^{n}))$, such that $M_{\theta
}[f
]=\int_{\mathbb{R}^{n}}f(x)m_{\theta }(dx)$.

Let $\bar{\mathcal{P}}=\{m_\theta: \theta\in\bar{\Theta}\}$ be the family of all probability measures on $(\mathbb{R}^{n},\mathcal{B}(\mathbb{R}^{n}))$ such that
\[
\int_{\mathbb{R}^n}f(x)m_{\theta}(dx)\leq \bar{\Lambda}[f], \ \ \forall f\in C_{2p-1,Lip}(\mathbb{R}^n).
\]
Then from the above result, we obtain that \begin{equation*}
\bar{\Lambda}[f ]=\sup\limits_{\theta\in \bar{\Theta}}\int_{\mathbb{R}^n}f(x)m_{\theta}(dx), \ \ \forall f\in C_{2p-1,Lip}(\mathbb{R}^n).
\end{equation*}
Now we prove that $\mathcal{\bar{P}}$ is weakly compact. Set $f_{i}%
(x)=(|x|-i)^{+}\wedge1$, it is easy to check that $f_{i}\subset C_{2p-1,Lip}%
(\mathbb{R}^{n})$ and $f_{i}\downarrow0$. Then by Lemma \ref{newds}, we obtain%
\[
\sup_{\theta \in \bar{\Theta}}m_{\theta}(\{|x|\geq i+1\})\leq \bar{\Lambda}%
[f_{i}]\downarrow0.
\]
Thus $\mathcal{\bar{P}}$ is tight. Let $m_{\theta_{i}}$, $i\geq1$, converge
weakly to $m$. Then by the definition of weak convergence, we can get for any
$f\in C_{2p-1,Lip}(\mathbb{R}^{n})$, $N>0$, $M>0$,%
\[
\int_{\mathbb{R}^{n}}(f(x)\wedge N)\vee(-M)m(dx)\leq \bar{\Lambda}[(f\wedge
N)\vee(-M)].
\]
Note that $f\vee(-M)-(f\wedge N)\vee(-M)\downarrow 0$ as $N\uparrow \infty$, then
by Lemma \ref{newds}, we can get%
\[
0\leq \bar{\Lambda}[f\vee(-M)]-\bar{\Lambda}[(f\wedge N)\vee(-M)]\leq
\bar{\Lambda}[f\vee(-M)-(f\wedge N)\vee(-M)]\downarrow0.
\]
Thus by the monotone convergence theorem under $m$, we obtain%
\[
\int_{\mathbb{R}^{n}}f(x)\vee(-M)m(dx)\leq \bar{\Lambda}[f\vee(-M)],
\]
which implies $\int_{\mathbb{R}^{n}}f(x)\vee(-M)m(dx)\in \mathbb{R}$.
Similarly, we can get $\int_{\mathbb{R}^{n}}f(x)m(dx)\leq \bar{\Lambda}[f]$.
Thus $m\in \mathcal{\bar{P}}$, which completes the proof.
\end{proof}

In the classical case, i.e., $\bar{\Lambda}[\cdot]$ is a linear expectation, it is easy to check that $\bar{\Theta}$ only has a single element $\theta_0$. In particular, the
probability measure $m_{\theta_0}$ is the unique invariant measure for the diffusion process $X$. Under the $G$-expectation framework, we can also give the following definition.
\begin{definition}
 A sublinear expectation $\tilde{\mathbb{E}}$
on $(\mathbb{R}^n, C_{2p,Lip}(\mathbb{R}^n))$ is said to be an  invariant expectation for the $G$-diffusion process $X$ if
\[
\tilde{\mathbb{E}}[\mathbb{\hat{E}}[f(X^x_t)]]=\mathbb{\tilde{E}}[f(x)] \ \ \text{for each $f\in C_{2p,Lip}(\mathbb{R}^n)$ \ and $t\geq0$}.
\]
The family of probability measures that represents $\tilde{\mathbb{E}}$ on $(\mathbb{R}^n, C_{2p-1,Lip}(\mathbb{R}^n))$
 is called
 invariant for the $G$-diffusion process $X$.
\end{definition}
\begin{remark}
{\upshape
For the invariant expectation $\mathbb{\tilde{E}}[\cdot]$, it corresponds to
the family of probability measures, which can be explained as the uncertainty
of the initial distribution. Given this uncertainty of the initial
distribution, the left-hand side of the equality in the above definiton can be
explained as the uncertainty of the distribution of $X_{t}$. Thus under the
invariant expectation $\mathbb{\tilde{E}}[\cdot]$, the distribution
uncertainty to the $G$-diffusion process $X$ is invariant in time.

}
\end{remark}

\begin{theorem}\label{HW5}
Assume \emph{(H1)} and \emph{(H2}) hold. Then there exists a unique invariant expectation $\tilde{\mathbb{E}}$ for
the $G$-diffusion process $X$. Moreover,
for each $f\in  C_{2p,Lip}(\mathbb{R}^n)$, we have
\[
\tilde{\mathbb{E}}[f]=\bar{\Lambda}[f].
\]
\end{theorem}
\begin{proof}
{\bf Existence:}
Denote $\bar{f}(x):=\mathbb{\hat{E}}[f(X^x_t)]$.
By Lemma \ref{HW2} and Theorem \ref{HW3}, we can find some constant $C_1$ such that
\begin{align*}
|\bar{f}(x)-\bar{f}(x^{\prime})|\leq & |\mathbb{\hat{E}}[f(X^x_t)]-\mathbb{\hat{E}}[f(X^{x^{\prime}}_t)]|\\
\leq &
K_f \mathbb{\hat{E}}[(1+|X^x_t|^{2p-1}+|X^{x^{\prime}}_t|^{2p-1})|X^{x}_t-X^{x^{\prime}}_t|]\\
\leq & C_1 \mathbb{\hat{E}}[(1+|X^x_t|^{2p-1}+|X^{x^{\prime}}_t|^{2p-1})^{\frac{2p}{2p-1}}]^{\frac{2p-1}{2p}}\mathbb{\hat{E}}[|X^{x}_t-X^{x^{\prime}}_t|^{2p}]^{\frac{1}{2p}}
\\
\leq &C_1\exp(-\eta t) (1+|x|^{2p-1}+|x^{\prime}|^{2p-1})|x-x^{\prime}|.
\end{align*}
Thus $\bar{f}(x)\in C_{2p,Lip}(\mathbb{R}^n)$. From Theorem \ref{HW3} and Lemma A.3 of \cite{HW}, we get
\begin{align*}
\bar{\Lambda}[\bar{f}]=\lim\limits_{s\rightarrow\infty}\hat{\mathbb{E}}[\bar{f}(X^x_s)]
&=\lim\limits_{s\rightarrow\infty}\hat{\mathbb{E}}[\hat{\mathbb{E}}[{f}(X^x_t)]_{x=X^x_s}]\\
&=\lim\limits_{s\rightarrow\infty}\hat{\mathbb{E}}[\hat{\mathbb{E}}[{f}(X^{s,x}_{s+t})]_{x=X^x_s}]\\
&=\lim\limits_{s\rightarrow\infty}\hat{\mathbb{E}}[\hat{\mathbb{E}}[{f}(X^{s,X^x_s}_{s+t})]]\\&
=\lim\limits_{s\rightarrow\infty}\hat{\mathbb{E}}[{f}(X^{x}_{t+s})]\\
&=\bar{\Lambda}[{f}],
\end{align*}
which concludes that $\bar{\Lambda}$
is an  invariant  expectation for the $G$-diffusion process $X$.

{\bf Uniqueness:}
Assume $\tilde{\Lambda}$ is also an invariant  expectation for the $G$-diffusion process $X$.
Then for each $f\in C_{2p,Lip}(\mathbb{R}^{n})$ and $t\geq 0$, we obtain
\begin{align*}
\tilde{\Lambda}[f]=\tilde{\Lambda}[\hat{\mathbb{E}}[f(X^{x}_t)]].
\end{align*}
By Theorem \ref{HW3}, there exists a constant $C_1$ such that
\begin{align*}|\bar{\Lambda}[f]-\hat{\mathbb{E}}[f(X^{x}_t)]|
 \leq C_1(1+|x|^{2p})\exp(-\eta t).
\end{align*}
Consequently, we derive that
\begin{align*}
|\bar{\Lambda}[f]-\tilde{\Lambda}[f]|\leq
\lim\limits_{t\rightarrow\infty}|\tilde{\Lambda}[\bar{\Lambda}[f]]-\tilde{\Lambda}[\hat{\mathbb{E}}[f(X^{x}_t)]]|
\leq C_1\lim\limits_{t\rightarrow\infty}\exp(-\eta t)\tilde{\Lambda}[(1+|x|^{2p})]=0,
\end{align*}
and this completes the proof.
\end{proof}

\begin{theorem}\label{HW12}
 Assume \emph{(H1)-(H2)} hold and  $\tilde{\mathbb{E}}$ is a sublinear expectation on  $(\mathbb{R}^n, C_{2p,Lip}(\mathbb{R}^n))$.
 If there exists a point $t_0>0$ such that,
 \[
\tilde{\mathbb{E}}[\mathbb{\hat{E}}[f(X^x_{t_0})]]=\mathbb{\tilde{E}}[f(x)], \ \ \forall f\in C_{2p,Lip}(\mathbb{R}^n),
\]
 then $\tilde{\mathbb{E}}$ is the unique invariant  expectation for $X$.
 \end{theorem}
\begin{proof}
Denote $\bar{f}(x):=\mathbb{\hat{E}}[f(X^x_{t_0})]$. Then using the same method as in the proof of Theorem \ref{HW5},
we have
\[
\mathbb{\tilde{E}}[\bar{f}(x)]=\tilde{\mathbb{E}}[\mathbb{\hat{E}}[\bar{f}(X^x_{t_0})]]=
\tilde{\mathbb{E}}[\mathbb{\hat{E}}[\mathbb{\hat{E}}[f(X^x_{t_0})]_{x=X^x_{t_0}}]]=\tilde{\mathbb{E}}[\mathbb{\hat{E}}[f(X^x_{2t_0})]].
\]
In a similar way, we obtain for each integer $n\geq 1$,
\[\mathbb{\tilde{E}}[f(x)]=\tilde{\mathbb{E}}[\mathbb{\hat{E}}[f(X^x_{nt_0})]].\]
Then by Theorem \ref{HW3}, we get
\[
\mathbb{\tilde{E}}[f(x)]=\lim\limits_{n\rightarrow\infty}\tilde{\mathbb{E}}[\mathbb{\hat{E}}[f(X^x_{nt_0})]]=\bar{\lambda}^f,
\]
which is the desired result.
\end{proof}

Now we give some examples of invariant measures.

\begin{example}{
\upshape
Assume that $b(0)=h_{ij}(0)=\sigma(0)=0$, then it is easy to check that $X^0_t=0$. Then by Lemma \ref{HW2}, we obtain
$\hat{\mathbb{E}}[|X^x_t|]\leq \exp(-\eta t)|x|$ for each $t\geq 0$.
In particular, we obtain that \[
\bar{\Lambda}[f]=\lim\limits_{t\rightarrow\infty}\mathbb{\hat{E}}[f(X^0_t)]=f(0), \ \ \forall f\in C_{2p-1,Lip}(\mathbb{R}^n).
\]
Thus \[
\bar{\Lambda}[f]=\int_{\mathbb{R}^n}f(x)\delta_0(dx), \ \ \forall f\in C_{2p-1,Lip}(\mathbb{R}^n),
\]
where $\delta_0$ is Dirac measure.
}
\end{example}

Consider the following Ornstein-Uhlenbeck process driven by $G$-Brownian motion: for each $x\in\mathbb{R}^d$,
\begin{align}
\label{HW6}
Y_t^x=x-\alpha\int^t_0 Y_s^xds+B_t,
\end{align}
where $\alpha>0$ is a given constant. It is obvious that assumption (H2) holds for each $p\geq 1$ in this case.

\begin{lemma}\label{HW7}
The invariant expectation for $G$-Ornstein-Uhlenbeck process $Y$ is the  $G$-normal distribution of $\sqrt{\frac{1}{2\alpha}}B_1$.
\end{lemma}
\begin{proof}
From the $G$-It\^{o} formula, we get
\[
Y^x_t=\exp(-\alpha t)x+\exp(-\alpha t)\int^t_0\exp(\alpha s) dB_s, \ \text{for all $t\geq0$.}
\]

For each integer $N$, denote $t^N_i=\frac{it}{N}$ with $0\leq i\leq N$ and
$h^N_s:=\exp(\alpha t^N_i)\mathbf{1}_{[t^N_i,t^N_{i+1}}(s).$
Then it is obvious that
\[
\lim\limits_{N\rightarrow\infty}\hat{\mathbb{E}}[\int^t_0|\exp(\alpha s)-h^N_s|^2ds]=0.
\]
Thus $\|\int^t_0\exp(\alpha s) dB_s-\int^t_0 h^N_s dB_s\|_{L^2_G}\rightarrow 0$
as $N\rightarrow \infty$.

Note that $\int^t_0 h^N_s dB_s=\sum\limits_{i=0}^{N-1}\exp(\alpha t^N_i)(B_{t^N_{i+1}}-B_{t^N_{i}})$. Then
we get $\int^t_0 h^N_s dB_s$ and $\sqrt{\sum\limits_{i=0}^N\exp(2\alpha t^N_i)(t^N_{i+1}-t^N_{i})}B_1$ are  identically distributed.
Consequently, for each   $p\geq 1$ and $f\in C_{p,Lip}(\mathbb{R}^d)$,
\begin{align*}
\hat{\mathbb{E}}[f(\int^t_0\exp(\alpha s) dB_s)]=\lim\limits_{N\rightarrow\infty}\hat{\mathbb{E}}[f(\int^t_0 h^N_s dB_s)]=&\lim\limits_{N\rightarrow\infty}\hat{\mathbb{E}}[f(\sqrt{\sum\limits_{i=0}^N\exp(2\alpha t^N_i)(t^N_{i+1}-t^N_{i})}B_1)]\\=&\hat{\mathbb{E}}[f(\sqrt{\int^t_0\exp(2\alpha s)ds}B_1)]\\
=&\hat{\mathbb{E}}[f(\sqrt{\frac{1}{2\alpha}(\exp(2\alpha t)-1)}B_1)].
\end{align*}
Thus,  for each $p\geq 1$ and $f\in C_{p,Lip}(\mathbb{R}^d)$ , we have
\[
\hat{\mathbb{E}}[f(\exp(-\alpha t)\int^t_0\exp(\alpha s) dB_s)]=\hat{\mathbb{E}}[f(\sqrt{\frac{1}{2\alpha}(1-\exp(-2\alpha t))}B_1)].
\]
Applying Lemma
\ref{HW2} yields that
\[
\lim\limits_{t\rightarrow\infty}\hat{\mathbb{E}}[f(Y^0_t)]=\lim\limits_{t\rightarrow\infty}\hat{\mathbb{E}}[f(\exp(-\alpha t)\int^t_0\exp(\alpha s) dB_s)]=\lim\limits_{t\rightarrow\infty}\hat{\mathbb{E}}[f(\sqrt{\frac{1}{2\alpha}(1-\exp(-2\alpha t))}B_1)]=\hat{\mathbb{E}}[f(\sqrt{\frac{1}{2\alpha}}B_1)].
\]
Thus by Theorem \ref{HW5}, we obtain
\[
\bar{\Lambda}[f]=\hat{\mathbb{E}}[f(\sqrt{\frac{1}{2\alpha}}B_1)],
\]
which is the desired result.
 \end{proof}

 \begin{example}{
\upshape
Suppose $B$ is a  $1$-dimensional $G$-Brownian motion. For each $x\in\mathbb{R}$, let
\begin{align*}
Y_t^x=x+\int^t_0(m- Y_s^x)ds+B_t+\langle B\rangle_t,
\end{align*}
where $m$ is a given constant. From the $G$-It\^{o} formula, we get
\[
Y^x_t=\exp(-t)x+m(1-\exp(-t))+\int^t_0\exp(s-t) dB_s+\int^t_0\exp(s-t) d\langle B\rangle_s, \ \text{for all $t\geq0$.}
\]
By a similar analysis as in Lemma \ref{HW7}, we obtain that $\int^t_0\exp(s-t) dB_s+\int^t_0\exp(s-t) d\langle B\rangle_s$ and $\sqrt{\frac{1}{2}(1-\exp(-2 t))}B_1+(1-\exp(-t))\langle B\rangle_1$ are identically distributed.
Then  for each $p\geq 1$ and $f\in C_{p,Lip}(\mathbb{R})$ , we have
\[
\hat{\mathbb{E}}[f(Y^x_t)]=\hat{\mathbb{E}}[f(m+\sqrt{\frac{1}{2}}B_1+\langle B\rangle_1)].
\]
}
\end{example}

Next we shall consider the following $G$-diffusion process:
for each $x\in\mathbb{R}$,
\begin{align}
\label{HW6}
Y_t^x=x-\alpha\int^t_0 Y_s^xd\langle B\rangle_s+B_t,
\end{align}
where $\alpha>0$ is a given constant.  Applying the $G$-It\^{o} formula, we get
\[
Y^x_t=\exp(-\alpha \langle B\rangle_t)x+\exp(-\alpha \langle B\rangle_t)\int^t_0\exp(\alpha \langle B\rangle_s) dB_s, \ \text{for all $t\geq0$.}
\]
From Theorems \ref{HW3}, \ref{HW12} and Lemma \ref{HW7}, we have the following.
\begin{corollary}
Given a sublinear space $(\mathbb{R}, C_{p,Lip}(\mathbb{R}), \tilde{\mathbb{E}})$ and denote $\zeta(x)=x$ for $x\in\mathbb{R}$, then
$\tilde{\mathbb{E}}$ is  the invariant measure for $G$-process $Y^x$
if and only if for some point $t> 0$ and $x\in\mathbb{R}$,
  $\exp(-\alpha \langle B\rangle_t)\zeta+\exp(-\alpha \langle B\rangle_t)\int^t_0\exp(\alpha \langle B\rangle_s) dB_s$ and $\zeta$ are identically distributed, where $(B_t)_{t\geq 0}$ is independent from $\zeta$.
\end{corollary}

\section{Ergodic measure}
In this section,  we shall only consider non-degenerate $G$-Brownian motion,
i.e., there exist some constants $\underline{\sigma}^2 >0$ such that, for any $A\geq B$ \[
G(A)-G(B)\geq  \frac{1}{2}\underline{\sigma}^2tr[A-B].\]

  We begin with the following lemma, which is essentially from \cite{HW}.
\begin{lemma}\label{HWm1}
Assume \emph{(H1)} and \emph{(H2)} hold. Then for each $f\in C_{2p,Lip}(\mathbb{R}^n)$,
the following fully nonlinear  ergodic PDE:
\begin{align} \label{HW11}
&G(H(D_{x}^{2}{v},D_{x}{v},x))+\langle
b(x),D_{x}{v}\rangle+f(x)=\lambda^f,
\end{align}%
has a solution $(v,\lambda^f)\in C_{2p,Lip}(\mathbb{R}^n)\times\mathbb{R}$,
where
\begin{align*}
H_{ij}(D_{x}^{2}{v},D_{x}{v},x)=  &  \langle D_{x}^{2}{v}\sigma_{i}%
(x),\sigma_{j}(x)\rangle+2\langle D_{x}{v},h_{ij}(x)\rangle.
\end{align*}
Moreover, if $(\bar{v},\bar{\lambda})\in C_{2p,Lip}(\mathbb{R}^n)\times\mathbb{R}$ is also a solution to equation \eqref{HW11}, then we have
\[\bar{\lambda}=\lambda^f=\lim\limits_{T\rightarrow\infty}\frac{1}{T}\hat{\mathbb{E}}[\int^T_0 f(X^x_s)ds], \ \ \ \forall x\in\mathbb{R}^n.\]
\end{lemma}
\begin{proof}
The proof is immediate from Lemma \ref{HW2}, Theorems 5.4 and 5.5 of \cite{HW}.
\end{proof}

Denote a mapping $\Lambda: C_{2p,Lip}(\mathbb{R}^n)\mapsto\mathbb{R}$ by \[\Lambda[f]=\lambda^f. \]
 By a similar analysis as in Lemma \ref{HW3},
it is easy to check that ${\Lambda}$ is a sublinear expectation on $(\mathbb{R}^n, C_{2p,Lip}(\mathbb{R}^n))$.
\begin{lemma}
Assume \emph{(H1)} and \emph{(H2)} hold. Then we obtain
\begin{description}
\item[(a)] If $f_1\geq f_2$, then ${\Lambda}[f_1]\geq
{\Lambda}[f_2]$;
\item[(b)]  ${\Lambda}[c]=c$ for each constant $c$;
\item[(c)]  ${\Lambda}[f_1+f_2]\leq{\Lambda}%
[f_1]+{\Lambda}[f_2]$;
\item[(d)] ${\Lambda}[\lambda f]=\lambda
{\Lambda}[f]$ for each $\lambda\geq0$.
\end{description}
\end{lemma}

In addition, we also have the following result.
\begin{theorem}
Assume  \emph{(H1)} and \emph{(H2)} hold. Then
there exists a family of weakly compact probability measures $\{m_{\theta}\}_{\theta\in \Theta}$ defined on $(\mathbb{R}^n, \mathcal{B}(\mathbb{R}^n))$ such that
\[
\Lambda[f]=\sup\limits_{\theta\in \Theta}\int_{\mathbb{R}^n}f(x)m_{\theta}(dx), \ \ \forall f\in C_{2p-1,Lip}(\mathbb{R}^n).
\]
\end{theorem}
\begin{proof}
The proof is similar to Theorem \ref{HW4}.
\end{proof}

\begin{definition}
 A sublinear expectation $\tilde{\mathbb{E}}$
on $(\mathbb{R}^n, C_{2p,Lip}(\mathbb{R}^n))$ is said to be an  ergodic expectation for the $G$-diffusion process $X$ if
\[
\tilde{\mathbb{E}}[f]=\lim\limits_{T\rightarrow\infty}\frac{1}{T}\hat{\mathbb{E}}[\int^T_0 f(X^x_s)ds], \ \ \forall f\in C_{2p,Lip}(\mathbb{R}^n).
\]
The family of probability measures that represents $\tilde{\mathbb{E}}$ is called ergodic for the $G$-diffusion process $X$.
\end{definition}

\begin{proposition}\label{HW9}
Let \emph{(H1)} and \emph{(H2)} hold. Then
for each $v\in C_{2p-1,Lip}(\mathbb{R}^n)$ with $\partial_{x_i}v\in C_{2p-2,Lip}(\mathbb{R}^n)$ and $\partial_{x_ix_j}^2v \in  C_{2p-3,Lip}(\mathbb{R}^n)$,
we have
\[
\Lambda[-G(H(D_{x}^{2}{v},D_{x}{v},x))-\langle
b(x),D_{x}{v}\rangle]=\sup\limits_{\theta\in \Theta}\int_{\mathbb{R}^n}[-G(H(D_{x}^{2}{v},D_{x}{v},x))-\langle
b(x),D_{x}{v}\rangle] m_{\theta}(dx)=0.
\]
\end{proposition}
\begin{proof}
Taking $f=-G(H(D_{x}^{2}{v},D_{x}{v},x))-\langle
b(x),D_{x}{v}\rangle$, by equation \eqref{HW11}, we obtain $\lambda^f=0$ and the proof is complete.
\end{proof}

\begin{example}{
\upshape
Assume that $b(0)=h_{ij}(0)=\sigma(0)=0$, then  we obtain that
\[
{\Lambda}[f]=\lim\limits_{T\rightarrow\infty}\frac{1}{T}\mathbb{\hat{E}}[\int^T_0f(X^0_t)dt]=f(0), \ \ \forall f\in C_{2p-1,Lip}(\mathbb{R}^n).
\]
Thus \[
\Lambda[f]=\bar{\Lambda}[f]=\int_{\mathbb{R}^n}f(x)\delta_0(dx), \ \ \forall f\in C_{2p-1,Lip}(\mathbb{R}^n).
\]
}
\end{example}

Note that $\hat{\mathbb{E}}[\int^T_0 f(X^x_s)ds]\leq \int^T_0 \hat{\mathbb{E}}[f(X^x_s)]ds$. Then it follows from Corollary \ref{HW8} that $\lambda^f\leq \bar{\lambda}^f$ and
$\Theta\subset\bar{\Theta}$. In the classical case, it is obvious  that $\Lambda=\bar{\Lambda}$. In particular, if $\bar{\Theta}$ only has a single element,
it is easy to check that $\lambda^f= \bar{\lambda}^f$.
However, in general we cannot get $\Lambda=\bar{\Lambda}$ under $G$-framework.
\begin{example}{
\upshape
Assuming $d=1$ and $0< \underline{\sigma}^2<\bar{\sigma}^2=1$.
Consider the following $G$-Ornstein-Uhlenbeck process:
for each $x\in\mathbb{R}$,
\begin{align}
\label{HW61}
Y_t^x=x-\frac{1}{2}\int^t_0 Y_s^xds+B_t.
\end{align}
Note that  $
Y^x_t=\exp(-\frac{1}{2} t)x+\exp(-\frac{1}{2} t)\int^t_0\exp(\frac{1}{2} s) dB_s $.
By Proposition \ref{HW9} and taking $v(x)=\frac{1}{2}x^4$,
 we have \[
 \Lambda[x^4-G(6x^2)]=\Lambda[x^4-3x^2]=0.
\]
It follows from Lemma \ref{HW7} that
$\bar{\Lambda}[x^4-3x^2]=\hat{\mathbb{E}}[B_1^4-3B_1^2]$.
Denote by $E_{\sigma}$ the linear expectation corresponding to the normal distributed density function
$N(0, \sigma^2)$ with $\underline{\sigma}^2\leq \sigma^2\leq 1$. Then for each  $p\geq 1$ and $f\in C_{p,Lip}(\mathbb{R})$ , $$\mathbb{\hat{E}}[f(B_1)]\geq \sup\limits_{\underline{\sigma}^2\leq \sigma^2\leq 1}E_{\sigma}[f(B_1)].$$
From the definition of $G$-expectation, we obtain that $\hat{\mathbb{E}}[B_1^4-3B_1^2]=\hat{\mathbb{E}}[\hat{\mathbb{E}}[(x+B_1-B_{\frac{1}{2}})^4-3(x+B_1-B_{\frac{1}{2}})^2]_{x=B_{\frac{1}{2}}}]$.
Set $g(x)=\hat{\mathbb{E}}[(x+B_1-B_{\frac{1}{2}})^4-3(x+B_1-B_{\frac{1}{2}})^2]$ and $g_1(x)={{E}}_1[(x+B_1-B_{\frac{1}{2}})^4-3(x+B_1-B_{\frac{1}{2}})^2]$,
$g_2(x)={{E}}_{\underline{\sigma}}[(x+B_1-B_{\frac{1}{2}})^4-3(x+B_1-B_{\frac{1}{2}})^2]$.
It is obvious that $g(x)\geq g_1\vee g_2(x).$
After direct calculus, we obtain
\[
g_1(x)=x^4-\frac{3}{4}, \ g_2(x)=x^4+3(\underline{\sigma}^2-1)x^2+\frac{3}{4}\underline{\sigma}^4-\frac{3}{2}\underline{\sigma}^2.
\]
Consequently,
\[
g_1\vee g_2(x)=g_1(x)\mathbf{1}_{|x|>\frac{\sqrt{1-\underline{\sigma}^2}}{2}}+g_2(x)\mathbf{1}_{|x|\leq \frac{\sqrt{1-\underline{\sigma}^2}}{2}}.
\]
Then we have
\begin{align*}
E_1[g_1\vee g_2(B_{\frac{1}{2}})]=&E_1[B_{\frac{1}{2}}^4-\frac{3}{4}\mathbf{1}_{|B_{\frac{1}{2}}|> \frac{\sqrt{1-\underline{\sigma}^2}}{2}}+(3(\underline{\sigma}^2-1)B_{\frac{1}{2}}^2+\frac{3}{4}\underline{\sigma}^4-\frac{3}{2}\underline{\sigma}^2)\mathbf{1}_{|B_{\frac{1}{2}}|\leq \frac{\sqrt{1-\underline{\sigma}^2}}{2}}]\\
=&3E_1[[\frac{1}{4}(1-\underline{\sigma}^2)^2-(1-\underline{\sigma}^2)B_{\frac{1}{2}}^2]\mathbf{1}_{|B_{\frac{1}{2}}|\leq \frac{\sqrt{1-\underline{\sigma}^2}}{2}}]
\\
\geq & 3E_1[[\frac{1}{4}(1-\underline{\sigma}^2)^2-(1-\underline{\sigma}^2)B_{\frac{1}{2}}^2]\mathbf{1}_{|B_{\frac{1}{2}}|\leq \frac{\sqrt{1-\underline{\sigma}^2}}{4}}]\\
\geq &\frac{9}{16} (1-\underline{\sigma}^2)^2 E_1[\mathbf{1}_{|B_{\frac{1}{2}}|\leq \frac{\sqrt{1-\underline{\sigma}^2}}{4}}]>0.
\end{align*}
Thus we get $\hat{\mathbb{E}}[B_1^4-3B_1^2]\geq E_1[g_1\vee g_2(B_{\frac{1}{2}})]>0$ and $\bar{\Lambda}[x^4-3x^2]\neq\Lambda[x^4-3x^2].$
}
\end{example}
\begin{example}{
\upshape
Assuming $d=1$ and $0< \underline{\sigma}^2<\bar{\sigma}^2=1$.
Let us consider equation \eqref{HW6} with $\alpha=\frac{1}{2}$. Under each linear expectation  $E_{\sigma}$ with $\underline{\sigma}^2\leq \sigma^2\leq 1$, it is easy to check that the invariant measure of equation \eqref{HW6} is the standard normal distributed
density function $E_1$.
However,  we claim that the invariant measure of equation \eqref{HW6} cannot be the normal distributed
density function $E_1$.
Otherwise, the ergodic measure of equation \eqref{HW6}  is also the normal distributed
density function $E_1$. Therefore, by Proposition \ref{HW9} and taking $v(x)=x^2$,
 we have \[
 \Lambda[-G(2-2x^2)]=E_1[-G(2-2B_1^2)]=E_1[(\underline{\sigma}^2(1-B^2_1)^--(1-B^2_1)^+)]=(\underline{\sigma}^2-1)E_1[(1-B^2_1)^+]\neq 0,
\]
which is a contradiction.
}
\end{example}

\begin{remark}{\upshape
Assume $d=1$ and  $b(x)=-x$, $h(x)=0$  and $\sigma=1$.
Then consider the following equation:
 \begin{align}\label{HW10}
 \begin{cases}
&\partial_tu-G(D_{x}^{2}u)+xD_{x}u=0,\ (t,x)\in(0,\infty)\times\mathbb{R},\\
& u(0,x)=f(x).
\end{cases}
\end{align}
Denote $\bar{u}(t,x):=\int^t_0 u(s,x)ds=\int^t_0\hat{\mathbb{E}}[f(X^x_s)]ds.$
Assume $u(s,x)$ is a smooth function.
Then \[
\partial_t\bar{u}(t,x)=u(t,x), \ \partial_x\bar{u}(t,x)=\int^t_0 \partial_xu(s,x)ds,\  \partial_{xx}^2\bar{u}(t,x)=\int^t_0 \partial_{xx}^2u(s,x)ds.
\]
In the linear case, i.e., $G(a)=\frac{1}{2}a$, it is easy to check that
 \begin{align*}
\partial_t\bar{u}-\frac{1}{2}D_{x}^{2}\bar{u}+xD_{x}\bar{u}+f=0.
 \end{align*}
Then by the ergodic theory, we obtain \[
\bar{\Lambda}[f]=\lim\limits_{T\rightarrow\infty}\frac{1}{T}\int^T_0{E}[f(X^x_s)]ds=\lim\limits_{T\rightarrow\infty}\frac{\bar{u}(T,x)}{T}=\Lambda[f].\]
However, under the nonlinear expectation framework, there is no such  relationship for fully nonlinear PDE \eqref{HW10}.
}
\end{remark}

\begin{remark}{\upshape
In the linear expectation case, ergodic theory and related problems are connected with the invariant measure. However, from the above results,
this relationship may not hold true under the nonlinear expectation framework. Thus we should study nonlinear ergodic problems via ergodic expectation $\Lambda$
instead of invariant expectation $\bar{\Lambda}$. In particular, \cite{HW} obtained the links between  ergodic expectation  and  large time behaviour of solutions to fully nonlinear PDEs.
}
\end{remark}


\begin{thebibliography}{99}


\bibitem{DZ} Da Prato, G. and Zabczyk, J. (1996) Ergodicity for infinite-dimensional systems. London Mathematical
Society Note Series, 229, Cambridge University Press, Cambridge.

\bibitem{DH}  Debussche, A.,   Hu, Y. and  Tessitore, G. (2011) Ergodic BSDEs under weak dissipative
assumptions. Stochastic Process. Appl., 121(3), 407-426.

\bibitem {DHP11}Denis, L., Hu, M. and Peng S. (2011) Function spaces and
capacity related to a sublinear expectation: application to $G$-Brownian
motion pathes. Potential Anal., 34, 139-161.



\bibitem{G} \textsc{Gao, F.} (2009) {Pathwise properties and homomorphic flows for stochastic differential equations driven by $G$-Brownian motion}. {Stochastic Processes and their Applications,} {119}, 3356-3382.



\bibitem{HJPS}  Hu, M.,  Ji, S., Peng, S. and Song, Y. (2014) {Backward stochastic differential equations
driven by $G$-Brownian motion}. Stochastic Processes and their Applications, 124, 759-784.

\bibitem{HJPS1} Hu, M.,  Ji, S., Peng, S. and Song, Y. (2014) {Comparison theorem, Feynman-Kac formula
and Girsanov transformation for BSDEs driven
by $G$-Brownian motion}. Stochastic Processes and their Applications, 124,  1170-1195.

\bibitem {HP09} Hu, M. and Peng, S. (2009) On representation theorem of
$G$-expectations and paths of $G$-Brownian motion. Acta Math. Appl. Sin. Engl.
Ser., 25(3), 539-546.

\bibitem {HW} Hu, M. and Wang, F. (2014) Ergodic BSDEs driven by $G$-Brownian motion and their  applications, arxiv:1407.6210.

\bibitem{H} Hu, Y., Madec, P.-Y. and Richou, A. (2014) Large time behaviour of mild solutions of Hamilton-Jacobi-Bellman equations in infinite dimension by a probabilistic approach, arxiv:1406.5993v1.

\bibitem {L-P}Li, X. and Peng, S. (2011) Stopping times and related
It\^o's calculus with $G$-Brownian motion. Stochastic Processes and their
Applications, 121, 1492-1508.

%

\bibitem{Peng2005} Peng, S. (2005) {Nonlinear expectations and
nonlinear Markov chains,} Chin. Ann. Math., 26B(2) 159-184.

\bibitem {P07a}Peng, S. (2007) $G$-expectation, $G$-Brownian Motion and
Related Stochastic Calculus of It\^o type. Stochastic analysis and
applications, 541-567, Abel Symp., 2, Springer, Berlin.

\bibitem {P08a}Peng, S. (2008) Multi-dimensional $G$-Brownian motion and
related stochastic calculus under $G$-expectation. Stochastic Processes and
their Applications, 118(12), 2223-2253.

\bibitem {P10}Peng, S. (2010) Nonlinear expectations and stochastic
calculus under uncertainty, arXiv:1002.4546v1.

\bibitem {PengICM2010}Peng, S. (2010) Backward stochastic differential
equation, nonlinear expectation and their applications, in Proceedings of the
International Congress of Mathematicians Hyderabad, India.
281-307.

\bibitem{R}  Richou, A. (2009) Ergodic BSDEs and related PDEs with Neumann boundary conditions. Stochastic Process. Appl., 119, 2945-2969.

\bibitem{RM} Royer, M. (2004) BSDEs with a random terminal time driven by a monotone generator and their links with PDEs. Stoch.
Stoch. Rep., 76(4), 281-307.

\end{thebibliography}
\end{document}